\documentclass[12 pt,twoside]{amsart}

\usepackage{amsopn}
\usepackage{amssymb}
\usepackage{amscd}

\newtheorem{theorem}{Theorem}[section]
\newtheorem{lemma}[theorem]{Lemma}
\newtheorem{proposition}[theorem]{Proposition}
\newtheorem{corollary}[theorem]{Corollary}

\theoremstyle{definition}
\newtheorem{definition}[theorem]{Definition}

\theoremstyle{remark}
\newtheorem{remark}[theorem]{Remark}

\numberwithin{equation}{section}

\begin{document}

\title[Coarse transitivity and coarsely $J$-class operators]{Coarse topological transitivity on open cones and coarsely $J$-class and $D$-class operators}

\author[A. Manoussos]{Antonios Manoussos}
\address{Fakult\"{a}t f\"{u}r Mathematik, SFB 701, Universit\"{a}t Bielefeld, Postfach 100131, D-33501 Bielefeld, Germany}
\email{amanouss@math.uni-bielefeld.de} \urladdr{http://www.math.uni-bielefeld.de/~amanouss}
\thanks{During this research the author was fully supported by SFB 701 ``Spektrale Strukturen und
Topologische Methoden in der Mathematik" at the University of Bielefeld, Germany.}


\date{}


\begin{abstract}
We generalize the concept of coarse hypercyclicity, introduced by Feldman in \cite{Fe1}, to that of coarse topological transitivity on open cones. We show that a bounded
linear operator acting on an infinite dimensional Banach space with a coarsely dense orbit on an open cone is hypercyclic and a coarsely topologically transitive
(mixing) operator on an open cone is topologically transitive (mixing resp.). We also ``localize" these concepts by introducing two new classes of operators called
coarsely $J$-class and coarsely $D$-class operators and we establish some results that may make these classes of operators potentially interesting for further studying.
Namely, we show that if a backward unilateral weighted shift on $l^2(\mathbb{N})$ is coarsely $J$-class (or $D$-class) on an open  cone then it is hypercyclic. Then we
give an example of a bilateral weighted shift on $l^{\infty}(\mathbb{Z})$ which is coarsely $J$-class, hence it is coarsely $D$-class, and not $J$-class. Note that,
concerning the previous result, it is well known that the space $l^{\infty}(\mathbb{Z})$ does not support $J$-class bilateral weighted shifts, see \cite{CosMa2}.
Finally, we show that there exists a non-separable Banach space which supports no coarsely $D$-class operators on  open cones. Some open problems are added.
\end{abstract}

\maketitle

\section{Introduction and basic concepts}

N.S. Feldman in \cite{Fe1} introduced the concept of coarse hypercyclicity under the name ``$d$-\textit{density}". A bounded linear operator $T:X\to X$ acting on a
separable Banach space $X$ is called \textit{coarsely hypercyclic} if it has an orbit within a bounded distance by a positive constant $d$ of every vector (the name
``$d$-\textit{density}" in \cite{Fe1} came from the constant $d$). Feldman showed that such an orbit may not be dense in $X$ although an operator with a coarsely dense
orbit is always hypercyclic, i.e. $T$ has a dense orbit, see \cite{Fe1,BaMa,GEPe}. Coarsely dense orbits appear naturally by looking at perturbations of a dense orbit by
a vector with bounded orbit \cite{Fe1}. In \cite{AbMa} we studied, together with H. Abels, a similar problem for finitely generated abelian subsemigroups of $GL(V)$,
where $V$ is a finite dimensional complex (or real) vector space. We showed that if such a semigroup has a coarsely dense orbit on an open  cone of $V$ then, for the
complex case, this orbit is actually dense in $V$. Recall that an \textit{open cone} on a Banach space $X$ is an open subset $C$ of $X$ such that $\lambda x\in C$ for
every $\lambda
>0$.

Motivated by the previous mentioned concepts and results we introduce and study the concept of coarse topological transitivity on open  cones. This concept can be seen
as a generalization of coarse hypercyclicity introduced by Feldman in \cite{Fe1}. An advantage of coarse topological transitivity comparing to coarse hypercyclicity is
that coarsely topologically transitive operators may exist also on non-separable Banach spaces like topologically transitive operators do. Recall that a bounded linear
operator $T:X\to X$ acting on a Banach space $X$ is called \textit{topologically transitive} if for every non-empty open subsets $U,V$ of $X$ there exists a non-negative
integer $n$ such that $T^nU\cap V\neq\emptyset$.

\begin{definition}\label{def11}
Let $X$ be an infinite dimensional Banach space. A bounded linear operator $T:X\to X$ is called \textit{coarsely topologically transitive} on an open  cone $C\subset X$
(with respect to a positive constant $d$) if for every non-empty open set $U\subset X$ and for every $x\in C$ there exists a non-negative integer $n$ such that $T^nU\cap
B(x,d)\neq\emptyset$, where $B(x,d)$ denotes the open ball centered at $x\in X$ with radius $d$. If $C=X$ we say that $T$ is coarsely topologically transitive.
\end{definition}

A bounded linear operator $T:X\to X$ acting on a Banach space $X$ is called \textit{topologically mixing} if for every non-empty open subsets $U,V$ of $X$ there exists a
positive integer $N$ such that $T^nU\cap V\neq\emptyset$ for every $n\geq N$. One can define a coarse analogue concept in the following way.

\begin{definition}\label{def12}
Let $X$ be an infinite dimensional Banach space. A bounded linear operator $T:X\to X$ is called \textit{coarsely topologically mixing} on an open  cone $C\subset X$
(with respect to a positive constant $d$) if for every non-empty open set $U\subset X$ and for every $x\in C$ there exists a non-negative integer $N$ such that $T^nU\cap
B(x,d)\neq\emptyset$ for every $n\geq N$. If $C=X$, $T$ is called coarsely topologically mixing.
\end{definition}

Our main result is the following theorem. Its proof will be given in several steps in Section 2.

\begin{theorem}\label{main}
Let $T:X\to X$ be a bounded linear operator acting on an infinite dimensional Banach space $X$. Then the following hold.
\begin{enumerate}
\item[(i)] If $T$ has a coarsely dense orbit on an open  cone then $T$ is hypercyclic.

\item[(ii)] If $T$ is coarsely topologically transitive on an open cone then $T$ is topologically transitive
on $X$.

\item[(iii)] If $T$ is coarsely topologically mixing on an open  cone then $T$ is topologically mixing
on $X$.
\end{enumerate}
\end{theorem}

In \cite{CosMa1}, together with G. Costakis, we ``localized" the concept of a topologically transitive operator by introducing the $J$-class operators. In a similar way,
in Section 3, we ``localize" the concept of a coarsely topologically transitive operator by introducing two new classes of operators called coarsely $D$-class operators
and coarsely $J$-class operators. Since we do not want to repeat ourselves, we establish some results that may make these classes of operators potentially interesting
for further studying. Firstly we show that if a backward unilateral weighted shift on $l^2(\mathbb{N})$ is coarsely $J$-class (or $D$-class) on an open  cone $C\subset
X$ then it is hypercyclic. Then we give an example of a bilateral weighted shift on $l^{\infty}(\mathbb{Z})$ which is coarsely $J$-class, hence it is coarsely $D$-class,
and not $J$-class. Note that, it is well known that the space $l^{\infty}(\mathbb{Z})$ does not support $J$-class bilateral weighted shifts, see \cite{CosMa2}. Finally,
we show that there exists a non-separable Banach space which supports no coarsely $D$-class operators on  open cones. In Section 4, we give some final remarks and we
raise three open problems.

\textit{Note that from now on when we say ``a linear operator" we always mean a bounded linear operator on a complex Banach space.} The closure of a subset $A$ of a
Banach space $X$ is denoted by $\overline{A}$ and its boundary by $\partial A$. The open ball centered at $x\in X$ with radius $d$ is denoted by $B(x,d)$.

The use of various limit sets, like  in \cite{BaMa,CosMa2,CosMa1,Ma1}, and their coarse variations will help us to make the proofs simpler and clearer for the reader. So
let us introduce some notation.

Let $T:X\to X$ be a linear operator acting on a Banach space $X$ and let $x\in X$. The \textit{orbit} of $x$ under $T$ is the set $O(x,T)=\{ T^nx:\, n\geq 0\}$. The
\textit{coarse orbit} of $x$ under $T$ with respect to a positive constant $d$ is the set

\[
\begin{split}
O(x,T,d)=\{ &y\in X:\,\,\mbox{there exists a non-negative integer}\,\,n\\
            &\mbox{such that}\,\,\| T^{n}x-y\|< d\}=\bigcup_{y\in O(x,T)}B(y,d).
\end{split}
\]

The orbit $O(x,T)$ is \textit{coarsely dense} on an open  cone $C\subset X$ with respect to a positive constant $d$ if $C\subset O(x,T,d)$.

The \textit{extended} (\textit{prolongational}) \textit{limit set} of $x$ under $T$ is the set

\[
\begin{split}
J(x,T)=\{ &y\in X:\,\,\mbox{there exist a strictly increasing sequence of }\\
&\mbox{positive integers}\,\,\{k_{n}\}\,\mbox{and a sequence}\,\,\{x_{n}\}\subset X\\
&\mbox{such that}\, x_{n}\rightarrow x\,\mbox{and}\,\,T^{k_{n}}x_{n}\rightarrow y\}
\end{split}
\]
and describes the asymptotic behavior of the orbits of vectors nearby to $x$. The corresponding coarse set is

\[
\begin{split}
J(x,T,d)=\{ &y\in X:\,\,\mbox{ there exist a strictly increasing sequence of}\\
&\mbox{positive integers}\,\,\{k_{n}\}\,\,\mbox{and a sequence}\,\,\{x_{n}\}\subset X\,\mbox{such}\\
&\mbox{that}\,\,x_{n}\rightarrow x\,\,\mbox{and}\,\, \| T^{k_n}x_n-y\|< d \,\,\mbox{for every}\, n\in\mathbb{N} \}.
\end{split}
\]
Note that in the definition of $J(x,T,d)$ we do not require the convergence of the sequence $\{T^{k_n}x_n\}_{n\in\mathbb{N}}$! Note also that $\bigcup_{y\in
J(x,T)}B(y,d)\subset J(x,T,d)$. The converse inclusion does not hold in general. It may happen that $J(x,T,d)=X$ and $J(x,T)=\emptyset$, see the example in Remark
\ref{rem32}.

The \textit{extended mixing limit set} of $x$ under $T$ is the set

\[
\begin{split}
J^{mix}(x,T)=\{ &y\in X:\,\,\mbox{there exists a sequence}\,\,\{x_{n}\}\subset X\,\mbox{such that}\\
&x_{n}\rightarrow x\,\,\mbox{and}\,\, T^{n}x_{n}\rightarrow y\}
\end{split}
\]

and the corresponding coarse set is

\[
\begin{split}
J^{mix}(x,T,d)=\{ &y\in X:\,\,\mbox{there exists a sequence}\,\,\{x_{n}\}\subset X\,\,\mbox{such that}\\
&x_{n}\rightarrow x\,\,\mbox{and}\,\, \| T^{n}x_n-y\|< d \,\,\mbox{for every}\,\, n\in\mathbb{N} \}.
\end{split}
\]

The \textit{prolongation of the orbit} $O(x,T)$ is the set $D(x,T):=O(x,T)\cup J(x,T)$ and the corresponding coarse set is  $D(x,T,d):=O(x,T,d)\cup J(x,T,d)$. The
\textit{mixing prolongation of the orbit} $O(x,T)$ is the set $D^{mix}(x,T):=O(x,T)\cup J^{mix}(x,T)$ and the corresponding coarse set is $D^{mix}(x,T,d):=O(x,T,d)\cup
J^{mix}(x,T,d)$.

\begin{remark}\label{rem01}
Given a vector $x\in X$ the sets $J(x,T)$, $J^{mix}(x,T)$ and $D(x,T)$ are $T$-invariant and closed in $X$, see \cite{CosMa2,CosMa1}.
\end{remark}

\begin{remark}\label{rem11}
It is plain to check that an operator $T:X\to X$ is topologically transitive if and only if $J(x,T)=D(x,T)=X$ for every $x\in X$ and $T$ is topologically mixing if and
only if $J^{mix}(x,T)=D^{mix}(x,T)=X$ for every $x\in X$. In view of the previous notation, it is also easy to see that an operator $T:X\to X$ is coarsely topologically
transitive on an open cone $C$ with respect to a positive constant $d$ if and only if $C\subset D(x,T,d)$, for every $x\in X$ and it is coarsely hypercyclic on $C$ if
and only if $C\subset O(x,T,d)$. Similarly for a coarsely mixing operator we have that $C\subset D^{mix}(x,T,d)$, for every $x\in X$. Note that if an operator $T$ is
coarsely topologically transitive on an open cone $C$ and since the coarse orbit $O(x,T,d)$ has always non-empty interior we can not deduce that $C\subset J(x,T,d)$ for
every $x\in X$. Hence, coarse $D$-sets come naturally into play in the place of coarse $J$-sets.
\end{remark}

The following proposition relates the various limit sets with their corresponding coarse sets.

\begin{proposition}\label{pr11}
Let $T:X\to X$ be a linear operator acting on an infinite dimensional Banach space $X$. Let $x,y\in X$, $d$ be a positive real number and let $\{ t_k\}\subset\mathbb{R}$
be a strictly increasing sequence of positive real numbers with $t_k\to +\infty$. Then the following hold.

\begin{enumerate}
\item[(i)] If $t_ky\in O(t_kx,T,d)$ for every $k\in\mathbb{N}$,  then $y\in \overline{O(x,T)}$.

\item[(ii)] If $t_ky\in D(t_kx,T,d)$ for every $k\in\mathbb{N}$, then $y\in D(x,T)$.

\item[(iii)] If $t_ky\in J(t_kx,T,d)$ for every $k\in\mathbb{N}$, then $y\in J(x,T)$.

\item[(iv)] If $t_ky\in J^{mix}(t_kx,T,d)$ for every $k\in\mathbb{N}$, then $y\in J^{mix}(x,T)$.

\item[(v)] If $t_ky\in D^{mix}(t_kx,T,d)$ for every $k\in\mathbb{N}$, then $y\in D^{mix}(x,T)$.
\end{enumerate}
\end{proposition}
\begin{proof}
We will only give the proof of item (iv) since the other items follow in a similar way. Assume that $t_ky\in J^{mix}(t_kx,T,d)$ for every $k\in\mathbb{N}$. Then, for
each positive integer $k\in \mathbb{N}$ there exist a positive integer $N_k$ and a sequence $\{ x_n^k\}_{n\in\mathbb{N}}\subset X$ such that $\| x_n^k- x\| <
\frac{1}{k}$ and $\| t_kT^n x_n^k - t_ky\| < d$ for every $n\geq N_k$. The sequence $\{ N_k\}_{k\in\mathbb{N}}$ can be chosen to be strictly increasing. For $N_k\leq n
<N_{k+1}$, $k=1,\ldots$ set $y_n:=x_n^k$  and $s_n:=t_k$. Fix an $\varepsilon >0$, a positive real number $M>0$ and $m\in\mathbb{N}$ such that $\frac{1}{m}
<\varepsilon$,  $t_m>M$ and $\frac{d}{t_m}<\varepsilon$. Now for every $n\geq N_m$ there exists a positive integer $k\geq m$ such that $N_m\leq N_k\leq n <N_{k+1}$. So,
$\| y_n -x\|=\|x_n^k -x\|< \frac{1}{k}\leq \frac{1}{m}<\varepsilon$ and $s_n=t_k\geq t_m>M$ since $\{ t_k\}$ is increasing. On the other hand, $\| s_nT^n y_n - s_ny\|=\|
t_kT^n x_n^k - t_ky\| < d$. Therefore, $\| T^n y_n - y\| < \frac{d}{s_n}=\frac{d}{t_k}\leq \frac{d}{t_m}<\varepsilon$, for every $n\geq N_m$. Since $m$ was fixed then
$y\in J^{mix}(x,T)$.
\end{proof}

\section{Coarsely hypercyclic and coarsely topological transitive operators on open cones}

This section is devoted to the proof of our main result Theorem \ref{main}. Unfortunately the intersection of two open coarsely dense subsets of a Banach space $X$ may
be empty. So, even if $X$ is separable, we can not deduce, by applying a Baire's category type theorem, that a coarsely topological transitive operator is coarsely
hypercyclic, although by Theorem \ref{main} the operator is actually hypercyclic. Thus, one must consider the two cases separately.

The proof of the following proposition follows the line of the proof of Feldman's theorem \cite[Theorem 2.1]{Fe1}.

\begin{proposition} \label{pr21}
Let $T:X\to X$ be a linear operator acting on an infinite dimensional Banach space $X$ with a coarsely dense orbit $O(x,T,d)$ on an open  cone $C\subset X$ with respect
to a positive constant $d$. Then the following hold.
\begin{enumerate}
\item[(i)] The vector $x$ is a cyclic vector for $T$.

\item[(ii)] The open  cone $C$ is contained in $O(\frac{M}{d}\, x,T,M)$ for every $M>0$.

\item[(iii)] Let $y$ be a vector in $C$ such that the ball $B(y,3d)$  is contained in $C$. Then, the set $O(x,T)\cap B(y,3d)$ contains infinitely many points.

\item[(iv)] The operator $T$ has dense range, and so all of its powers.
\end{enumerate}
\end{proposition}
\begin{proof}
(i) Fix a vector  $y\in C$ and a positive integer $n\in\mathbb{N}$. Since $ny\in C\subset O(x,T,d)$, there exists a non-negative integer $k$ such that $\| T^k x-ny \|
<d$. Hence, $\| \frac{1}{n} T^kx-y \| <d/n \to 0$. Therefore, the closed linear span of the orbit $O(x,T)$ contains $C$, so it is the whole space $X$.

(ii) For every $y\in C$ the vector $\frac{d}{M}\,y\in C$. Hence, there exists a non-negative integer $n$ such that $\| T^nx- \frac{d}{M}\,y\| <d$. Thus, $\|
T^n(\frac{M}{d}\,x)- y\| <M$ and the proof is finished.

(iii) Since $X$ is infinite-dimensional, the boundary of the open ball $B(y,2d)$ is not compact. Hence, there exists an infinite sequence $\{y_n\}$ in the boundary of
$B(y,2d)$ with the property whenever $n\neq k$ then $B(y_n,d)\cap B(y_k,d)=\emptyset$. Since $O(x,T)\cap B(y_n,d)\neq\emptyset$ the set $O(x,T)\cap B(y,3d)$ contains
infinitely many points.

(iv) It is enough to show that the open cone $C$ is contained in the closure of the range of $T$. Take a non-zero vector $y\in C$ and fix a positive integer
$n\in\mathbb{N}$. Then, by item (ii), for $M=\frac{1}{n}$, there exists a non-negative integer $k_n$ such that $\| T^{k_n}(\frac{1}{nd}x)- y\| <\frac{1}{n}$. Since $y$
is non-zero we may assume that $k_n>0$ for every $n\in\mathbb{N}$ and the proof is finished.
\end{proof}

\begin{theorem}\label{th21}
Let $T:X\to X$ be a linear operator acting on an infinite dimensional Banach space $X$ with a coarsely dense orbit $O(x,T,d)$ on an open  cone $C\subset X$ with respect
to a positive constant $d$. Then $T$ is a hypercyclic operator.
\end{theorem}
\begin{proof}
The proof is given in two steps. Firstly we show that $C$ is contained in $J(y,T)$ for every $y\in C$. Then we show that $C$ contains a cyclic vector for $T$ and hence
we may apply an extended Bourdon-Feldman theorem \cite[Theorem 4.1 and Corollary 4.6]{CosMa1} to deduce that $T$ is hypercyclic.

Let $y\in C$. We will show that $C\subset J(y,T)$. Take a point $w\in C$ and let $U\subset C$, $V\subset C$ be open neighborhoods of $y$ and $w$ respectively. Let $M>0$
be such that $B(y,3M)\subset U$ and $B(w,3M)\subset V$. By Proposition \ref{pr21} (ii) $C\subset O(\frac{M}{d}\, x,T,M)$. Thus, by item (iii) of the same proposition,
the sets $O(\frac{M}{d}\, x,T)\cap B(y,3M)$ and $O(\frac{M}{d}\, x,T)\cap B(w,3M)$ contain infinite many points. Hence, there exist a positive integer $n$ such that
$T^n(\frac{M}{d}\, x)\in U$ and a positive integer $m>n$ such that $T^m(\frac{M}{d}\, x)\in V$. So, $T^{m-n}U\cap V\neq\emptyset$, hence $w\in J(y,T)$.

The coarse orbit $O(x,T,d)$ intersects the open  cone $C$ infinitely many times. Hence, there exists a positive integer $N$ such that $T^Nx\in C$. By Proposition
\ref{pr21} (i), $x$ is a cyclic vector for $T$ and from item (iv) of the same proposition the operator $T^N$ has dense range. Hence, by \cite[Lemma 4.2]{CosMa1}, the
vector $T^Nx$ is a cyclic vector for $T$ and, since  $C\subset J(T^Nx,T)$, the extended limit set $J(T^Nx,T)$ has non-empty interior. Then, by the main result of
\cite[Theorem 4.1 and Corollary 4.6]{CosMa1}, $T$ is a hypercyclic operator.
\end{proof}

The proof of next lemma is quite similar to the proof of  \cite[Corollary 2.3 (i) and Corollary 2.5]{Ma1} and is omitted.

\begin{lemma}\label{lem21}
Let $X$ be a completely metrizable space and let $T:X\to X$ be a continuous map acting on $X$. Assume that there is a countable set $A$ and an open ball $B(y,\varepsilon
)$, for some $y\in X$ and $\varepsilon >0$, such that $x\in D(w,T)$ for every $x\in A$ and $w\in B(y,\varepsilon )$. Then, there is a dense $G_{\delta}$ subset $E$ of
$\overline{B(y,\varepsilon )}$ such that $\overline{A}\subset \overline{O(z,T)}$ for every $z\in E$.
\end{lemma}

\begin{remark}\label{rem21}
In the second step of the proof of Theorem \ref{th21} we used an extended Bourdon-Feldman theorem \cite[Theorem 4.1 and Corollary 4.6]{CosMa1}. We may show that $T$ is
hypercyclic using the original Bourdon-Feldman theorem \cite{BF}, i.e. that a somewhere dense orbit is dense in $X$. Indeed, by the first step of the proof of Theorem
\ref{th21}, we get that $C\subset J(y,T)$ for every $y\in C$. By Proposition \ref{pr21} (i), the operator $T$ is cyclic, thus $X$ is separable. So $C$ is also separable,
that is $C$ can be written as the closure of a countable subset of $C$. Therefore, by Lemma \ref{lem21}, there exists a vector in $C$ such that the closure of its orbit
contains $C$. Hence, by \cite{BF}, the operator $T$ is hypercyclic.
\end{remark}

As we mentioned in the introduction, Feldman in \cite{Fe1} showed that a coarsely hypercyclic operator $T:X\to X$ is actually hypercyclic but a coarsely dense orbit may
not be dense in $X$. In the following proposition we give a simple criterion under which a coarsely dense orbit on an open cone is dense in $X$.

\begin{proposition}\label{pr22}
Let $T:X\to X$ be a linear operator acting on an infinite dimensional Banach space $X$ with a coarsely dense orbit $O(x,T,d)$ on an open  cone $C\subset X$ for some
$x\in X$, $d>0$. Then, the orbit $O(x,T)$ is dense in $X$ if and only if $\lambda x\in \overline{O(x,T)}$ for some $| \lambda |<1$.
\end{proposition}
\begin{proof}
If $\lambda x\in \overline{O(x,T)}$ for some $| \lambda |<1$ then $\lambda^n x\in \overline{O(x,T)}$, for every positive integer $n$. Indeed, since $\lambda x\in
\overline{O(x,T)}$ there exists an increasing sequence of non-negative integers $\{ k_n\}$ such that $T^{k_n}x\to \lambda x$. Note that $T^{k_n} (\lambda x)\to \lambda^2
x$ and since $\lambda x\in \overline{O(x,T)}$ and $\overline{O(x,T)}$ is $T$-invariant then $\lambda^2 x \in \overline{O(x,T)}$. Now, proceeding by induction, we get
that $\lambda^n x\in \overline{O(x,T)}$ for every positive integer $n$.

Take a vector $y\in C$. We will show that $y\in \overline{O(x,T)}$. Since the orbit $O(x,T)$ is coarsely dense on the open cone $C\subset X$, i.e. $C\subset O(x,T,d)$,
then for every $n\in\mathbb{N}$ there exists a non-negative integer $k_n$ such that $\| T^{k_n} x - \frac{1}{\lambda^n} y \| < d$. Thus, $\| T^{k_n} (\lambda^n x) - y \|
< \lambda^n d$ and since $| \lambda |<1$ then $T^{k_n} (\lambda^n x)\to y$. Now, note that $\lambda^n x\in \overline{O(x,T)}$ and since $\overline{O(x,T)}$ is
$T$-invariant then $T^{k_n} (\lambda^n x) \in \overline{O(x,T)}$ for every $n\in\mathbb{N}$, so $y\in \overline{O(x,T)}$. Thus, $C\subset\overline{O(x,T)}$, hence by
\cite{Fe1}, $\overline{O(x,T)}=X$.
\end{proof}

In the following we proceed with the proof of the second and the third part of Theorem \ref{main}. The second part of Theorem \ref{main} will follow as a corollary of a
more general statement, see Theorem \ref{ballth} below.

\begin{proposition} \label{denserange}
Let $T:X\to X$ be a linear operator acting on an infinite dimensional Banach space $X$ such that the set $\bigcap_{x\in X}D(x,T)$ has non-empty interior. Then, the
operator $P(T)$ has dense range for every non-zero polynomial $P$ over the complex numbers.
\end{proposition}
\begin{proof}
It suffices to show that the operator $T-\lambda I$ has dense range, where $I:X\to X$ denotes the identity operator and $\lambda\in\mathbb{C}$. Assume the contrary,
hence, by the Hahn-Banach theorem, there exists a non-zero linear functional $x^{*}$ such that $x^{*}((T-\lambda I)(x))=0$ for every $ x\in X$. Therefore,
$x^{*}(T^{n}x)={\lambda}^{n}x^{*}(x)$ for every $x\in X$ and $n>0$. Since we assumed that $\bigcap_{x\in X}D(x,T)$ has non-empty interior there exists a non-empty open
set $U\subset D(x,T)=O(x,T)\cup J(x,T)$ for every $x\in X$. Therefore, $U\subset \bigcap_{x\in X}J(x,T)$ since $X$ has no isolated points.

Assume that $|\lambda|\leq 1$. Take an arbitrary vector $y\in U$ and fix a vector $x\in X$. Since $y\in U\subset \bigcap_{x\in X}J(x,T)\subset J(x,T)$ there exist a
sequence $x_n\to x$ and a strictly increasing sequence $\{ k_{n} \}$ of positive integers such that  $T^{k_{n}} x_{n}\rightarrow y$. So, if $|\lambda|< 1$ then
$x^*(y)=0$ since $x^{*}(T^{k_n}x_n)={\lambda}^{k_n}x^{*}(x_n)$. Thus, $x^*(U)=\{0\}$ which contradicts the open mapping theorem. Now, if $|\lambda|= 1$ then $x^*(y)\in
\overline{\{ {\lambda}^{n}x^{*}(x)\,|\, n>0\}}$. Hence, $x^*(U)\subset \overline{\{ {\lambda}^{n}x^{*}(x)\,|\, n>0\}}$ which again contradicts the open mapping theorem.

Now assume that $|\lambda|>1$. Since $x^{*}$ is non-zero there exists a vector $x\in X$ such that $x^{*}(x)\neq 0$. Take a vector $y\in U\subset \bigcap_{x\in
X}J(x,T)\subset J(x,T)$. Hence, there exist a sequence $x_n\to x$ and a strictly increasing sequence $\{ k_{n} \}$ of positive integers such that  $T^{k_{n}}
x_{n}\rightarrow y$. Since $x^{*}(x_{n})=\frac{1}{{\lambda}^{k_{n}}}\, x^{*}(T^{k_{n}}x_{n})$ then $x^{*}(x)=0$ which is a contradiction.
\end{proof}

\begin{lemma} \label{cyclicJ}
Let $T:X\to X$ be a  linear operator acting on an infinite dimensional Banach space $X$. If $y\in X$ is an interior point of $\bigcap_{x\in X}D(x,T)$ then $P(T)y\in
\bigcap_{x\in X}D(x,T)$ for every non-zero polynomial $P$ over the complex numbers.
\end{lemma}
\begin{proof}
Since $y\in \bigcap_{x\in X}D(x,T)$ then, $P(T)y\in \bigcap_{x\in X}D(P(T)x,T)$, for every $P$ non-zero polynomial. By Proposition \ref{denserange} the set
$\{P(T)x\,|\,x\in X\}$ is dense in $X$ and using an argument similar to that in \cite[Lemma 2.5]{CosMa2} we get that $P(T)y\in \bigcap_{x\in X}D(x,T)$.
\end{proof}

\begin{theorem}\label{ballth}
Let $T:X\to X$ be a linear operator acting on an infinite dimensional Banach space $X$ such that the set $\bigcap_{x\in X}D(x,T)$ has non-empty interior. Then $T$ is
topologically transitive on $X$.
\end{theorem}
\begin{proof}
Recall that $T:X\to X$ is topologically transitive if $D(x,T)=X$, for every $x\in X$. Fix two vectors $x,y\in X$. We will show that $y\in D(x,T)$. By an argument similar
to that in \cite[Lemma 2.5]{CosMa2} it is enough to show that for every open neighborhood $U$ of $y$ there exists a vector $v\in U$ such that $v\in D(x,T)$. Fix an open
neighborhood $U$ of $y$ and consider a vector $s_1$ in the interior $S$ of the set $\bigcap_{x\in X}D(x,T)$. Then, by Lemma \ref{lem21}, there is a vector $v\in U$ such
that $s_1\in \overline{O(v,T)}$. Let $W_v$ be the closed linear span of the orbit $O(v,T)$, then $s_1\in W_v\cap S\neq\emptyset$. Therefore, using again Lemma
\ref{lem21}, there exists a vector $s\in S$ such that $W_v\cap S\subset \overline{O(s,T)}$. So $s_1\in W_v\cap S\subset W_s$, where $W_s$ denotes the closed linear span
of the orbit $O(s,T)$. Note that the set $W_v\cap S$ is a non-empty relatively open subset of $W_v$ and since $W_v\cap S\subset W_s$ we get that $W_v\subset W_s$. The
vector $s\in S$ is an interior point of the set $\bigcap_{x\in X}D(x,T)$. Thus, by Lemma \ref{cyclicJ}, $P(T)s\in D(x,T)$ for every non-zero polynomial $P$ over the
complex numbers, so $W_s\subset D(x,T)$ since the set $D(x,T)$ is closed in $X$. Therefore, $v\in W_v\subset W_s\subset D(x,T)$ and the proof is finished.
\end{proof}

\begin{theorem} \label{th22}
Let $T:X\to X$ be a linear operator acting on an infinite dimensional Banach space $X$. Then the following hold.
\begin{enumerate}
\item[(i)] If $T$ is coarsely topologically transitive on an open  cone $C\subset X$, then $T$ is topologically transitive
on $X$.

\item[(ii)] If $T$ is coarsely topologically mixing on an open  cone $C\subset X$, then $T$ is topologically mixing
on $X$.
\end{enumerate}
\end{theorem}
\begin{proof}
(i) Since $T$ is coarsely topologically transitive on the open  cone $C\subset X$ then, by Proposition \ref{pr11} (ii), $C\subset D(x,T)$ for every $x\in X$. Now the
proof follows directly from Theorem \ref{ballth}.

(ii) If $T$ is coarsely topologically mixing on an open  cone $C\subset X$ then, by Proposition \ref{pr11} (v), $C\subset D^{mix}(x,T)$ for every $x\in X$. Since $X$ has
no isolated points then $C\subset J^{mix}(x,T)$ for every $x\in X$. Especially, $0\in J^{mix}(x,T)$ for every $x\in X$, since $J^{mix}(x,T)$ is a closed subset of $X$.
It is also plain to see that $J^{mix}(0,T)$ is a closed linear subspace of $X$ and since $C\subset J^{mix}(0,T)$ then $J^{mix}(0,T)=X$. Now fix a vector $x\in X$. We
will show that $J^{mix}(x,T)=X$. Since $0\in J^{mix}(x,T)$ there exists a sequence $x_n\to x$ such that $T^nx_n\to 0$. Take a vector $y\in X=J^{mix}(0,T)$. Thus, there
exists a sequence $y_n\to 0$ such that $T^ny_n\to y$. Hence, $T^n(x_n+y_n)\to y$ while $x_n+y_n\to x$. Therefore, $y\in J^{mix}(x,T)$ and the proof is finished.
\end{proof}

\begin{corollary}\label{cor21}
Let $T:X\to X$ be a linear operator acting on an infinite dimensional Banach space $X$. The following hold.
\begin{enumerate}
\item[(i)] if $T$ is coarsely topologically transitive then it is topologically transitive.

\item[(ii)] if $T$ coarsely topologically mixing then it is topologically mixing.
\end{enumerate}
\end{corollary}
\begin{proof}
Both of the items follow directly from Theorem \ref{th22} but they can also be derived from Proposition \ref{pr11}. Indeed, let $T$ be coarsely topologically transitive
(or coarsely topologically mixing) operator with respect to a positive constant $d$, then $D(x,T,d)=X$ (or $D^{mix}(x,T,d)=X$ respectively), for every $x\in X$. Hence,
by Proposition \ref{pr11} items (ii) and (v), $D(x,T)=X$ (or $D^{mix}(x,T)=X$ respectively), for every $x\in X$.
\end{proof}

\section{Coarsely $J$-class and $D$-class operators}

There is a ``standard" way to ``localize" a concept given via a global property of one of the various concepts of limit sets of a linear operator and this is to ask for
the existence of a non-zero vector with the same limit set property. As an example, in \cite{CosMa1} together with G. Costakis, we ``localized" the concept of
topologically transitive operators by introducing the $J$-class operators. Recall that an operator $T:X\to X$ acting on a Banach space is called $J$-class if there
exists a non-zero vector $x\in X$ such that $J(x,T)=X$. Note that if $T$ is topologically transitive then $D(x,T)=O(x,T)\cup J(x,T)=X$ for every $x\in X$ and since $X$
has no isolated points this is equivalent to $J(x,T)=X$ for every $x\in X$. For the coarse case things differ since a coarse orbit has always non-empty interior. So, we
may define two new classes of operators, the \textit{coarsely $J$-class} and the \textit{coarsely $D$-class} operators by requiring the existence of a non-zero vector
$x\in X$ with $J(x,T,d)=X$ or $D(x,T,d)=X$ respectively for some positive constant $d$. Coarsely $J$-class operators can also be derived by looking at perturbations of a
$J$-vector, i.e. a vector $x\in X$ with $J(x,T)=X$ by a vector with bounded orbit and similarly for $D$-class operators. As we mentioned in the introduction, in this
section we establish some results that may make these classes of operators potentially interesting for further studying. Namely,

\begin{enumerate}
\item[$\bullet$] A backward unilateral weighted shift on $l^2(\mathbb{N})$ is coarsely $J$-class (or $D$-class) on an open  cone then it is hypercyclic.

\item[$\bullet$] There is a bilateral weighted shift on $l^{\infty}(\mathbb{Z})$ which is coarsely $J$-class, hence it is coarsely $D$-class, and not $J$-class. Note
that, it is well known that the space $l^{\infty}(\mathbb{Z})$ does not support $J$-class bilateral weighted shifts, see \cite{CosMa2}.

\item[$\bullet$] There exists a non-separable Banach space which supports no coarsely $D$-class operators on  open cones.
\end{enumerate}

\begin{remark}\label{rem31}
In \cite{AzMu} Azimi and M\"uller constructed a linear operator $T:l_1\to l_1$  so that $J(0,T)$ has non-empty interior and $J(0,T)\neq l_1$. Since when $x\in J(0,T)$
then $\lambda x\in J(0,T)$ for every $\lambda >0$, it is plain to see that in this case the extended limit set $J(0,T)$ contains an open  cone and at the same time
$J(0,T)\neq l_1$. Hence, by \cite[Remark]{AzMu}, there is a linear operator $L:X\to X$ acting on a Banach space and a non-zero vector $x\in X$ such that $J(x,L)$
contains an open  cone and $J(x,L)\neq X$.
\end{remark}

\begin{proposition}\label{pr31}
Every coarsely $J$-class(or $D$-class) backward unilateral weighted shift on $l^2(\mathbb{N})$ on an open  cone $C\subset l^2(\mathbb{N})$ is hypercyclic.
\end{proposition}
\begin{proof}
Let $T:l^2(\mathbb{N})\to l^2(\mathbb{N})$ be a backward unilateral weighted shift which is coarsely $J$-class (or $D$-class) on an open cone $C\subset l^2(\mathbb{N})$
with respect to a positive constant $d$. That is there exists a non-zero vector $x\in l^2(\mathbb{N})$ such that $C\subset J(x,T,d)$ (or $C\subset D(x,T,d)$ resp.). Fix
a point $y\in C$ and a positive integer $N$. Then, there exist a strictly increasing (or increasing resp.) sequence of positive (or non-negative resp.) integers
$\{k_{n}\}$ and a sequence $\{x_n\}$ such that $x_{n}\rightarrow x$ and $\| T^{k_n}x_n-Ny\|_2< d$ for every $n\in\mathbb{N}$. Hence,  $\| T^{k_n}\frac{x_n}{N}-y\|_2<
\frac{d}{N}$. Using a diagonal procedure we can find a strictly increasing (or increasing resp.) sequence of positive integers $m_n$ and a sequence $y_n\to 0$ such that
$T^{m_n}y_n\to y$. Therefore, $C\subset J(0,T)$ (or $C\subset D(0,T)=J(0,T)\cup O(0,T)=J(0,T)$  resp.), thus by \cite[Proposition 5.13]{CosMa1}, $T$ is hypercyclic. Note
that we have not used that the vector $x$ is non-zero.
\end{proof}

\begin{proposition}\label{pr32}
There exists a backward bilateral weighted shift on $l^{\infty}(\mathbb{Z})$ which is coarsely $J$-class and not $J$-class.
\end{proposition}
\begin{proof}
Let $T:l^{\infty}(\mathbb{Z})\to l^{\infty}(\mathbb{Z})$ be the backward bilateral weighted shift  with weight sequence $(\alpha_n )_{n\in\mathbb{Z}}$, $\alpha_n=2$
for $n\geq 1$ and $\alpha_n=1$ for $n\leq 0$. As we showed, with G. Costakis, in \cite[Remark 3.5]{CosMa2},  $J(0,T)=l^{\infty}(\mathbb{Z})$ and $T$ is not a
$J$-class operator. Now, consider the vector $y=\{y(n)\}_{n\in\mathbb{Z}} \in l^{\infty}(\mathbb{Z})$ with $y(n)=0$ for every $n\neq 0$ and $y(0)=1$. Obviously $\|
T^ny\|_{\infty}=1$ for every $n\in\mathbb{N}$. We will show that $J(y,T,2)=l^{\infty}(\mathbb{Z})$. Take a vector $x\in l^{\infty}(\mathbb{Z})$. Since
$J(0,T)=l^{\infty}(\mathbb{Z})$ there exist a strictly increasing sequence of positive integers $\{k_{n}\}$ and a sequence $\{x_n\}$ such that $x_{n}\rightarrow x$
and $T^{k_n}x_n\to x$. It is plain to see that $\| T^{k_n}(x_n+y) -x\|_{\infty}<2$, for $n\in\mathbb{N}$ big enough, so  $J(y,T,2)=l^{\infty}(\mathbb{Z})$.
\end{proof}

\begin{remark}\label{rem32}
Note that, in the above example, $J(y,T,2)=l^{\infty}(\mathbb{Z})$ but $J(y,T)=\emptyset$. Indeed, let us assume the contrary, that is there exists a vector
$w=\{w(n)\}_{n\in\mathbb{Z}}\in J(y,T)$. Hence, there exist a sequence $y_n\to y$ and a strictly increasing sequence of positive integers $\{k_{n}\}$ such that
$T^{k_n}y_n\to w$.  Thus, there exists a positive integer $n_0$ such that $\| y_n -y\|_{\infty} <\frac{1}{4}$ and  $\| T^{k_n}y_n - w\|_{\infty} <\frac{1}{4}$ for every
$n\geq n_0$. For the economy of the proof, let us denote the vector $y_n=\{y(n,k)\}_{k\in\mathbb{Z}}$. Now, since the 0-th coordinate of $y$ is 1 and all the rest
coordinates are 0, then $|w(-k_{n_1}) -y(n_1,0)| < \frac{1}{4}$ and $|y(n_1,0)-1| < \frac{1}{4}$. Therefore, $|w(-k_{n_1})-1| < \frac{1}{2}$ (1). On the other hand $\|
T^{k_{n_0}}y_{n_0} - w\|_{\infty} <\frac{1}{4}$, hence $|w(-k_{n_1})-y(n_0,k_{n_0}-k_{n_1})| < \frac{1}{4}$. Note that $|y(n_0,k_{n_0}-k_{n_1})| < \frac{1}{4}$. So,
$|w(-k_{n_1})| < \frac{1}{2}$ which if it is combined with (1) leads to a contradiction.
\end{remark}

\begin{theorem} \label{existmain}
There exists a non-separable Banach $X_A$ space which supports no coarsely $D$-class operators on  open cones.
\end{theorem}

In \cite{Amir} A. Bahman Nasseri gave a negative answer to a question, we asked together with G. Costakis in \cite{CosMa1}, whether every non-separable Banach space
supports a $J$-class operator. He used the complexification of a \textit{non-separable} Banach space $X_A$ constructed by S.A. Argyros, A.D. Arvanitakis and A.G. Tolias
in \cite{AAT} which has the property that every bounded linear operator on $X_A$ is of the form $T=z I+S$, where $S:X_A\to X_A$ is a strictly singular operator on $X_A$
with \textit{separable range}, $I:X_A\to X_A$ denotes the identity map on $X_A$ and $z\in\mathbb{C}$. \textit{We claim that such an operator can not be coarsely
$D$-class on an open  cone} (hence it can not be a coarsely $J$-class operator on an open cone). Of course one may think to use the famous Argyros-Haydon space
\cite{ArHa} in which all bounded linear operators are of the form $T=z I+K$ where $K$ is a compact operator, but this space is separable and by a theorem proved
independently by Ansari \cite{Ansa} and Bernal \cite{Bernal} such a space supports a mixing, hence hypercyclic operator. Before we proceed with the proof of our claim
let us say a few words about strictly singular operators. A \textit{strictly singular} operator $S:X\to Y$ between two Banach spaces $X$ and $Y$ is a bounded linear
operator such that there is no infinite dimensional closed subspace $Z$ of $X$ such that $S:Z\to S(Z)$, the restriction of $S$ to $Z$, is an isomorphism. For example,
every compact operator is strictly singular. If $X=Y$ then the spectrum of $S$ is countable and 0 is the only possible accumulation point (for more information about
strictly singular operators see e.g. \cite{AbAli}).

The proof of our claim is given in several steps.

\begin{proposition} \label{existprop1}
Let $T:X\to X$ be a linear operator acting on a Banach space $X$. Then the following hold:
\begin{enumerate}
\item[(i)] If the spectrum  $\sigma(T)$ of $T$ is contained in the open unit disk $\mathbb{D}\subset\mathbb{C}$ then $B(0,d)\subset J(x,T,d) \subset
\overline{B(0,d)}$ for every $x\in X$ and $d>0$.

\item[(ii)] If $\sigma(T)\subset \mathbb{C}\setminus\overline{\mathbb{D}}$ then $J(x,T,d)=\emptyset$ for every non-zero vector $x\in X$ and $d>0$. Moreover,
$J(0,T,d)=J^{mix}(0,T)=X$.

\item[(iii)] If $D(x,T,d)$ contains an open cone for some non-zero vector $x\in X$ and $d>0$. Then $\sigma(T)\cap \partial \mathbb{D}\neq\emptyset$.
\end{enumerate}
\end{proposition}
\begin{proof}
(i) Assume that $\sigma(T)\subset\mathbb{D}$. Then by the spectral radius formula of Gelfand $\| T^n\|\to 0$. Let $y\in J(x,T,d)$ for some $x\in X$, $d>0$. Hence, there
exist a strictly increasing sequence of positive integers $\{ k_n\}$ and a sequence $x_n\to x$ such that $\| T^{k_n}x_n-y\|<d$ for every $n\in\mathbb{N}$. Since $\|
T^n\|\to 0$ and $x_n\to x$ then $T^{k_n}x_n\to 0$, thus $y\in \overline{B(0,d)}$ and so $J(x,T,d) \subset \overline{B(0,d)}$. Let now $y\in B(0,d)$. Since $T^nx\to 0$
then $\|T^nx-y\|\to \|y\|<d$, hence $y\in J(x,T,d)$.

(ii) If $\sigma(T)\subset \mathbb{C}\setminus\overline{\mathbb{D}}$ then $T$ is invertible and $\| T^{-n}\|\to 0$. We argue by contradiction. Assume that there exists a
vector $y\in J(x,T,d)$ for some non-zero vector $x\in X$ and some $d>0$. Thus, there exist a strictly increasing sequence of positive integers $\{ k_n\}$ and a sequence
$x_n\to x$ such that $\| T^{k_n}x_n-y\|<d$ for every $n\in\mathbb{N}$. Hence, $\| x_n-T^{-k_n}y\| \leq \| T^{-k_n}\|\cdot \|T^{k_n}x_n-y\|< \| T^{-k_n}\|\cdot d$. Since
$\| T^{-k_n}\|\to 0$ then $T^{-k_n}y\to 0$ and  since $\| x_n-T^{-k_n}y\| < \| T^{-k_n}\|\cdot d$ then $x_n\to 0$. Thus, $x=0$ which is a contradiction. Let us now show
that $J(0,T,d)=J^{mix}(0,T)=X$. Since $J^{mix}(0,T)\subset J(0,T,d)$ it is enough to show that $J^{mix}(0,T)=X$. This follows easily by noticing that $x=T^n(T^{-n}x)=x$
while $T^{-n}x\to 0$ for every $x\in X$.

(iii) We argue by contradiction. Assume that $D(x,T,d)$ contains an open  cone $C\subset X$ for some non-zero vector $x\in X$ and $d>0$ and that $\sigma(T)\cap \partial
\mathbb{D}=\emptyset$. Thus we may apply the Riesz Decomposition theorem and decompose $X$ into two closed $T$-invariant subspaces $X_1$ and $X_2$ of $X$ such that
$X=X_1\oplus X_2$, $\sigma (T_1)=\{ \lambda\in\sigma (T)\,|\, |\lambda |<1\}$ and $\sigma (T_2)=\{ \lambda\in\sigma (T)\,|\, |\lambda |>1\}$ where $T_1$, $T_2$ denote
the restriction of $T$ on $X_1$ and $X_2$ respectively. Let $x=x_1+x_2$ with $x_1\in X_1$ and $x_2\in X_2$. It is easy to see that $D(x,T,d)\subset D(x_1,T_1,d) +
D(x_2,T_2,d)$. Note that, by item (i), the set $J(x_1,T_1,d)$ is bounded and since  $\| T_1^n\|\to 0$, the coarse orbit $O(x_1,T_1,d)$ is also bounded. Take a vector
$c=c_1+c_2\in C$ where $c_1\in X_1$ and $c_2\in X_2$. Since $C\subset D(x,T,d)$ and $D(x_1,T_1,d)=O(x_1,T_1,d)\cup J(x_1,T_1,d)$ is bounded, then for large $\lambda >0$,
$\lambda c_1=0$ hence $c_1=0$. Thus, $C\subset D(x_2,T_2,d)\subset X_2$, therefore, $X_2=X$ and so $T_2=T$. In that case, by item (ii), $J(x,T,d)=\emptyset$, hence
$C\subset O(x,T,d)$. Then,  by Theorem \ref{th21}, the operator $T$ is hypercyclic. Since $\sigma(T)\subset \mathbb{C}\setminus\overline{\mathbb{D}}$ this leads to a
contradiction (see e.g. \cite[Proposition 5.3]{GEPe}).
\end{proof}

\begin{proposition} \label{existprop2}
let $X$ be a non-separable Banach space and let $T:X\to X$ be a linear operator of the form $T=U+S$, where $U:X\to X$ is power bounded and $S:X\to X$ is an operator with
separable range. Then $T$ can not be a coarsely $D$-class operator on an open cone of $X$.
\end{proposition}
\begin{proof}
Note that the multiplication of an operator on the left or on the right by an operator with separable range gives an operator with separable range too. Hence, for every
$n\in\mathbb{N}$, the operator $T^n$ can be written as a sum of the form $T^n=U^n +S_n$ where $S_n:X\to X$ is a linear operator with separable range. We argue by
contradiction. Assume that there exists an open cone $C$ of $X$ such that $C\subset D(x,T,d)$ for some $x\in X$, $d>0$ and let $y\in C$. Thus, there exist an increasing
sequence of non-negative integers $\{ k_n\}$ and a sequence $x_n\to x$ such that $\| T^{k_n}x_n-y\|<d$ for every $n\in\mathbb{N}$. Since $T^{k_n}=U^{k_n} +S_{k_n}$ for
every $n\in\mathbb{N}$ then $T^{k_n}x_n=U^{k_n}x_n - U^{k_n}x+ U^{k_n}x + S_{k_n}x_n$. Since $U$ is power bounded and $x_n\to x$ we get that $U^{k_n}x_n - U^{k_n}x\to
0$. So if we fix some $\varepsilon>0$ then $T^{k_n}x_n\in B(0,\varepsilon)+ O(x,U)+ \bigcup_{n\in\mathbb{N}} S_{k_n}(X)$. The set $O(x,U)+ \bigcup_{n\in\mathbb{N}}
S_{k_n}(X)$ has a countable dense subset hence it can be covered by a countable family $\{ B(w_n,\varepsilon )\}_{n\in\mathbb{N}}$  of open balls of radius $\varepsilon
>0$. Hence, $T^{k_n}x_n\in B(0,\varepsilon)+ O(x,U)+ \bigcup_{n\in\mathbb{N}} S_{k_n}(X)\subset \bigcup_{n\in\mathbb{N}} B(w_n,2\varepsilon )$. Note that $y\in B(T^{k_n}x_n,d)$ for
every $n\in\mathbb{N}$. Therefore, $C\subset D(x,T,d)\subset \bigcup_{n\in\mathbb{N}} B(w_n,d+2\varepsilon )$. We will show that the open  cone $C$ has a countable base
hence it is separable. In this case every open ball of $X$ is separable and since $X$ is connected then $X$ is separable which is a contradiction. Let us now show that
the family $\{ B(\frac{w_n}{k},\frac{d+2\varepsilon}{k} )\}_{n,k\in\mathbb{N}}$ form a (countable) base for $C$. To see that, take an open ball $B(v,R)$ centered at a
vector $v\in C$ with radius $R>0$ and a positive integer $k$ such that $\frac{d+2\varepsilon}{k}<\frac{R}{2}$. Since $kv\in C \subset \bigcup_{n\in\mathbb{N}}
B(w_n,d+2\varepsilon )$ there exists a positive integer $n\in\mathbb{N}$ such that $kv\in B(w_n,d+2\varepsilon )$. Thus, $v\in B(\frac{w_n}{k},\frac{d+2\varepsilon}{k}
)$ and since $\frac{d+2\varepsilon}{k}<\frac{R}{2}$ then $ B(\frac{w_n}{k},\frac{d+2\varepsilon}{k} )\subset B(v,R)$ and the proof is finished.
\end{proof}

\begin{proof}[Proof of Theorem \ref{existmain}]
Recall that a linear operator $T:X_A\to X_A$ on $X_A$ is of the form $T=z I+S$, where $S:X_A\to X_A$ is a strictly singular operator on $X_A$ with separable range,
$I:X_A\to X_A$ denotes the identity map on $X_A$ and $z\in\mathbb{C}$. We argue by contradiction. Assume that there exists an open cone $C$ of $X$ such that $C\subset
D(x,T,d)$ for some non-zero vector $x\in X_A$ and $d>0$. We consider two cases depending on the modulus of $z\in\mathbb{C}$.

The case $|z|\leq 1$ follows directly from Proposition \ref{existprop2}, since $S$ has dense range.

Let $|z|>1$. Since $z$ is the only possible accumulation point of the spectrum $\sigma (T)$ of $T$ and $|z|>1$ then we can write $\sigma (T)$ as a disjoint union of the
closed sets $\sigma_1=\{ \lambda\in\sigma (T)\,|\, |\lambda |\leq 1\}$ and $\sigma_2=\{ \lambda\in\sigma (T)\,|\, |\lambda |> 1\}$ and then apply the Riesz Decomposition
theorem. Therefore, there exist two closed $T$-invariant subspaces $X_1$ and $X_2$ of $X_A$ such that $X_A=X_1\oplus X_2$ and if $T_1$, $T_2$ denote the restriction of
$T$ on $X_1$ and $X_2$, respectively, then $\sigma (T_1)=\sigma_1$ and $\sigma (T_2)=\sigma_2$. Note that by Proposition \ref{existprop1} (iii), $\sigma(T)\cap \partial
\mathbb{D}\neq\emptyset$, hence $X_1\neq \{ 0\}$. We claim that $X_1$ is a finite dimensional vector space. If not then the restriction $S:X_1\to S(X_1)$ of the strictly
singular operator $S$ on $X_1$ is again strictly singular, hence it can not be a linear isomorphism. Thus, $0\in\sigma(S|X_1)$ therefore
$z\in\sigma(T_1)=z+\sigma(S|X_1)$, a contradiction. Thus, $X_1$ is finite dimensional. Write the vector $x=x_1+x_2$ with $x_1\in X_1$ and $x_2\in X_2$ and note that
$C\subset D(x,T,d)\subset D(x_1,T_1,d) + D(x_2,T_2,d)$. The projection of the open cone $C$ on $X_1\neq \{ 0\}$ is again an open cone in $X_1$ which is contained in
$D(x_1,T_1,d)$. This implies, by \cite{CosMa3},  that $x_1=0$ and $\sigma_1\subset \mathbb{C}\setminus\overline{\mathbb{D}}$ which is a contradiction since
$\sigma_1\subset\overline{\mathbb{D}}$.

\end{proof}

\section{Final remarks and open problems}

\subsection{}
One of the most important theorems in the theory of hypercyclic operators is the Bourdon-Feldman Theorem \cite{BF}. Bourdon-Feldman theorem says that whenever an orbit
$O(x,T)$ is somewhere dense in $X$ then it is dense in $X$. Since every coarse orbit $O(x,T,d)$ has always non-empty interior, and in view of our main result Theorem
\ref{main}, it is natural to ask the following question.

\medskip

\noindent\textbf{Question 1.}\,\, If a coarse orbit $O(x,T,d)$ contains an open  cone $C\subset X$ is it true that $O(x,T,d)=X$?

\bigskip

A question related to the previous one is the following.

\medskip

\noindent\textbf{Question 2.}\,\, Obviously a coarsely $J$-class operator is coarsely $D$-class. Is the inverse implication true?

\subsection{}
Let $T:X\to X$ be a linear operator acting on a Banach space $X$. If $X$ is finite dimensional then  $\bigcup_{y\in J(x,T)}B(y,d) = J(x,T,d)$. If $X$ is infinite
dimensional this is not longer true. In Proposition \ref{pr32} we showed that there exists a backward bilateral weighted shift on $l^{\infty}(\mathbb{Z})$ which is
coarsely $J$-class and not $J$-class and we found a vector $y\in l^{\infty}(\mathbb{Z})$ such that $J(y,T,2)=X$ and $J(y,T)=\emptyset$, see Remark \ref{rem32}. Hence it
is natural to ask the following question.

\medskip

\noindent\textbf{Question 3.}\,\, Let $T:X\to X$ be a linear operator acting on an infinite dimensional Banach space $X$. Assume that $\bigcup_{y\in J(x,T)}B(y,d)=X$. Is
then true that $J(x,T)=X$?

\end{document}